\documentclass[a4paper,11pt,english,british,reqno]{article}
\usepackage{setspace,graphicx,
epstopdf,amsmath,amsfonts,amsgen, mathtools,
amstext,amsthm,amsbsy,amsopn,amssymb,
bbm,tikz,
parskip,verbatim,mathrsfs,enumerate,
xcolor,comment}  
\usepackage[utf8]{inputenc}   
\usepackage[top=2.5cm, bottom=2.5cm, left=3cm, right=3 cm]{geometry}
\usepackage[square,numbers]{natbib} 
\usepackage{todonotes}
\usepackage[pdfborder={0 0 0}]{hyperref}

\let\oldsqrt\sqrt
\def\sqrt{\mathpalette\DHLhksqrt}
\def\DHLhksqrt#1#2{%
\setbox0=\hbox{$#1\oldsqrt{#2\,}$}\dimen0=\ht0
\advance\dimen0-0.2\ht0
\setbox2=\hbox{\vrule height\ht0 depth -\dimen0}%
{\box0\lower0.4pt\box2}}

\def\be{\boldsymbol{e}}

\def\b0{\boldsymbol{0}}

 \newcommand{\C}     {\mathbb{C}} 
\newcommand{\R}     {\mathbb{R}} 
\newcommand{\Z}     {\mathbb{Z}} 
\newcommand{\N}     {\mathbb{N}} 
\renewcommand{\P}   {\mathbb{P}}

\newcommand{\Q}     {\mathbb{Q}} 
\newcommand{\T}     {\mathbb{T}}

\newcommand{\Acal}   {{\mathcal A }}

\newcommand{\Ecal}   {{\mathcal E }}

\newcommand{\Scal}   {{\mathcal S }}

\newcommand{\Wcal}   {{\mathcal W }}

\newcommand{\tr}{{\operatorname {Tr}}}
\newcommand{\Exp}{\mathscr{E}\kern-0.2mm{\operatorname{xp}}}
\newcommand{\Log}{\mathscr{L}\kern-0.2mm{\operatorname{og}}}

\def\1{{\mathchoice {1\mskip-4mu\mathrm l}      
{1\mskip-4mu\mathrm l} 
{1\mskip-4.5mu\mathrm l} {1\mskip-5mu\mathrm l}}}

\numberwithin{equation}{section}
\numberwithin{figure}{section}
\newtheoremstyle{plain}
  {6pt}
  {4pt}
  {\slshape}
  {}
  {\bfseries}
  {.}
  {0.5em}
  {}%
\newtheorem{thm}{\protect\theoremname}

  \newtheorem{rem}[thm]{\protect\remarkname}
  
  \newtheorem{lem}[thm]{\protect\lemmaname}
  \numberwithin{thm}{section}

\usepackage{babel}
  \addto\captionsbritish{\renewcommand{\corollaryname}{Corollary}}
  \addto\captionsbritish{\renewcommand{\definitionname}{Definition}}
  \addto\captionsbritish{\renewcommand{\factname}{Fact}}
  \addto\captionsbritish{\renewcommand{\propositionname}{Proposition}}
  \addto\captionsbritish{\renewcommand{\remarkname}{Remark}}
  \addto\captionsbritish{\renewcommand{\theoremname}{Theorem}}
  \addto\captionsenglish{\renewcommand{\corollaryname}{Corollary}}
  \addto\captionsenglish{\renewcommand{\definitionname}{Definition}}
  \addto\captionsenglish{\renewcommand{\factname}{Fact}}
  \addto\captionsenglish{\renewcommand{\propositionname}{Proposition}}
  \addto\captionsenglish{\renewcommand{\remarkname}{Remark}}
  \addto\captionsenglish{\renewcommand{\theoremname}{Theorem}}
  \providecommand{\corollaryname}{Corollary}
  \providecommand{\definitionname}{Definition}
  \providecommand{\factname}{Fact}
  \providecommand{\propositionname}{Proposition}
  \providecommand{\remarkname}{Remark}
\providecommand{\theoremname}{Theorem}
\providecommand{\lemmaname}{Lemma}

\begin{document}

\title{Site-monotonicity properties for reflection positive measures with applications to quantum spin systems}
\author{Benjamin Lees\footnote{School of Mathematics, University of Bristol, UK. Email: \textit{benjamin.lees@bristol.ac.uk}} \and Lorenzo Taggi\footnote{Weierstrass Institute for Applied Analysis and Stochastics, Berlin. Email: \textit{lorenzo.taggi@wias-berlin.de}}}
\date{}

\maketitle


\begin{abstract} 
We  consider a general statistical mechanics model on a product of local spaces and prove that, 
 if the corresponding measure is reflection positive, then
several site-monotonicity properties  for the two-point function hold.
As an application of such a general theorem, we derive site-monotonicity properties for the spin-spin correlation of the quantum Heisenberg antiferromagnet and $XY$ model, we prove that such spin-spin correlations are point-wise uniformly positive on vertices with all odd coordinates -- improving previous  positivity results which hold for the Ces\`aro sum -- and we derive site-monotonicity properties for the probability that a loop connects two vertices in various random loop models, including the loop representation of the spin O(N) model, the double-dimer model, the loop O(N) model,  lattice permutations, 
thus extending  the previous results of \textit{Lees and Taggi (2019)}.
\end{abstract}


\section{Introduction}\label{sec:intro}
We consider a general probabilistic model on the torus $\T_L=\Z^d/L\Z^d$,
whose realisations live in a product of local spaces. Each local space is associated to one of the vertices of $\T_L$ and elements of the local spaces interact with each other according to a probability measure.
 Such a general setting includes  various important models in statistical mechanics, for example  the spin O(N) model, the quantum Heisenberg anti-ferromagnet and $XY$ model, the dimer and the double-dimer model, lattice permutations, and the loop O(N) model. 
We prove that, if a linear functional acting on functions of our state space is \emph{reflection positive}, then several site-monotonicity properties for the two-point function hold. This generalises the monotonicity and positivity results of \cite{L-T} to a very general system.
This general result has the following implications.

Firstly, in their seminal paper  \cite{F-S-S}, Fr\"ohlich, Simon and Spencer introduced a method for proving the non-decay of correlations of the two-point function of several statistical mechanics models in dimension $d > 2$. 
This method was further developed in   \cite{F-I-L-S} and used in many other research works  (we additionally refer to \cite{B} for an overview).
More precisely, this method is used to prove that  the Ces\`aro sum of the two-point function is uniformly positive.
 Our general monotonicity result shows that,
when this method works, a stronger result can often be obtained. Namely not only is
the Ces\`aro sum of the two-point function uniformly positive in the system size, but the two-point function is also uniformly positive \emph{point-wise} for a positive fraction of vertices.
This result was derived by Lees and Taggi in \cite{L-T} in a special case  and here it is generalised to an abstract statistical mechanics setting.

As an example of a new application we consider quantum spin systems including the Heisenberg antiferromagnet and XY model, which were not covered by the framework of \cite{L-T}. Quantum spin systems are important class of statistical mechanics models whose realisation space is the tensor product of local Hilbert spaces and can be `represented' as systems of random interacting loops,  we refer to  \cite{U2} for an overview. It is already known \cite{D-L-S, F-I-L-S, K-L-S,K-L-S2} that the Gibbs states of this model are reflection positive in the presence of anti-ferromagnetic interactions and that, in dimension $d > 2$,  the Ces\`aro sum of the two-point function is uniformly positive  for large enough values of the inverse temperature parameter and system size. Our result implies that the spin-spin correlation is point-wise uniformly positive for vertices with all odd coordinates, extending the existing results. We fully expect that this uniform positivity should extend to all vertices, not just `odd' vertices.

Our third main result involves a general class of random loop soup models, which we refer to as the random path model. This class includes the loop representation of the  spin O(N) model \cite{ B-U, L-T}, the double-dimer model \cite{Kenyon}, lattice permutations \cite{B-T, T}, and the loop O(N) model \cite{P-S}.  In \cite{L-T}, site-monotonicity properties for the two-point function -- which is defined as the ratio of partition functions with a walk connecting two-points in a system of loops and the partition function with only loops -- were derived.  Here we extend the result to a general class of two-point functions, including the probability that two fixed vertices have a loop passing through both of them.




\section{Model and main result}\label{sec:model}

Consider the torus $\T_L=\Z^d/L\Z^d$ with $d\geq 2$ and $L\in 2\N$.
Denote by $o=(0,\dots,0)$ the origin of the torus. For each $x\in \T_L$ let $\Sigma_x$ be a Polish space of local states (for example $\mathbb{S}^{N-1}$, $\C^{2S+1}$, $\{-1,+1\}$,...). Further let $\otimes$ be some associative product between the $\Sigma_x$'s (for example the cartesian product or the tensor product). Our state space is
\begin{equation}
\Scal=\otimes_{x\in \T_L} \Sigma_x.
\end{equation}
We denote elements of $\Scal$ by $w=(w_x)_{x\in\T_L}$ where $w_x\in \Sigma_x$.
Let $\Acal_L$ be a real, finite dimensional, algebra of functions on $\Scal$ with unit (for example if $\Sigma_x=\mathbb{S}^{N-1}$ then we could take the cartesian product and $\Acal_L$ to be the algebra of functions $\Scal\to \R$ that are measurable with respect to the Haar measure on $\Scal$). Further, let $\langle\cdot\rangle$ be a linear functional on $\Acal_L$ such $\langle 1\rangle=1$. Our key requirement is that $\langle\cdot\rangle$ is \emph{reflection positive}, which we describe briefly.

\subsection{Reflection Positivity}
Consider a plane $R=\{z\in \R^d\, :\, z\cdot \be_i =m\}$ for some $m\in \tfrac12\Z\cap [0,L)$ and some $i\in\{1,\dots, d\}$. Let $\vartheta:\T_L\to\T_L$ be the reflection operator that reflects vertices of $\T_L$ in the plane $R$. More precisely, for any $x=(x_1,\dots,x_d)\in \T_L$
\begin{equation}
\vartheta(x)_k:=\begin{cases} x_ k & \text{if }k\neq i, \\ 2m-x_k\mod L & \text{if }k=i. \end{cases}
\end{equation}
If $m\in \tfrac12\Z\setminus \Z$ we call such a reflection a \emph{reflection through edges}, if $m\in\Z$ we call such a reflection a \emph{reflection through vertices}. We denote by $\T_L^+,\T_L^-$ the partition of $\T_L$ into two halves with the property that $\vartheta(\T_L^{\pm})=\T_L^{\mp}$.

We say a function $A\in\Acal_L$ has domain $D\subset \T_L$ if for any $w_1,w_2\in \Scal$ that agree on $D$ we have $A(w_1)=A(w_2)$. 
Consider the algebras $\Acal_L^+,\Acal_L^-\subset\Acal_L$, of functions with domain $\T_L^+,\T_L^-$ respectively. The reflection $\vartheta$ acts on elements $w\in \Scal$ as $(\vartheta w)_x=w_{\vartheta x}$ and for $A\in\Acal^+_L$ it acts as $\vartheta A(w)=A(\vartheta w)$.

We say that $\langle\cdot\rangle$ is \emph{reflection positive} with respect to $\vartheta$ if, for any $A,B\in\Acal^+_L$, 
\begin{enumerate}
\item $\langle A\vartheta B\rangle=\langle B\vartheta A\rangle $,
\item $\langle A\vartheta A\rangle\geq 0$.
\end{enumerate}
A consequence of this is the Cauchy-Schwarz inequality
\begin{equation}\label{eq:refpos}
\langle A\vartheta B\rangle^2\leq \langle A\vartheta A\rangle\langle B\vartheta B\rangle.
\end{equation}
We say $\langle\cdot\rangle$ is \emph{reflection positive for reflections through edges resp. vertices} if, for any reflection $\vartheta$ through edges resp. vertices, $\langle\cdot\rangle$ is reflection positive with respect to $\vartheta$. 

\subsection{Main results}
For $j\in\{1,2\}$ let $F^j_o\in\Acal_L$ be functions with domain $\{o\}$.
Fix an arbitrary site $x \in \T_L$ and let  $o=t_0$, $t_1$, $\ldots$, $t_k = x$
be a self-avoiding nearest-neighbour path from $o$ to $t$,
and for any $i \in \{1, \ldots, k\}$, let $\Theta_i$ be the reflection with respect to the plane going through the edge $\{  t_{i-1}, t_{i}  \}$.
Define
$$
(F^j_o)^{[x]} : = \Theta_k  \circ \Theta_{k-1} \, \ldots \, \circ \Theta_1 \, ( F^j_o ).
$$
Observe that the function $(F^j_o)^{[x]}$ does not depend on the chosen path (See  Figure \ref{Fig:refexample} for an illustration).
For a lighter notation denote by  $F^j_x=(F^j_o)^{[x]}$  the function obtained from $F^j_o$ by applying a sequence of reflections that send $o$ to $x$. 
\begin{figure}
\includegraphics[scale=0.26]{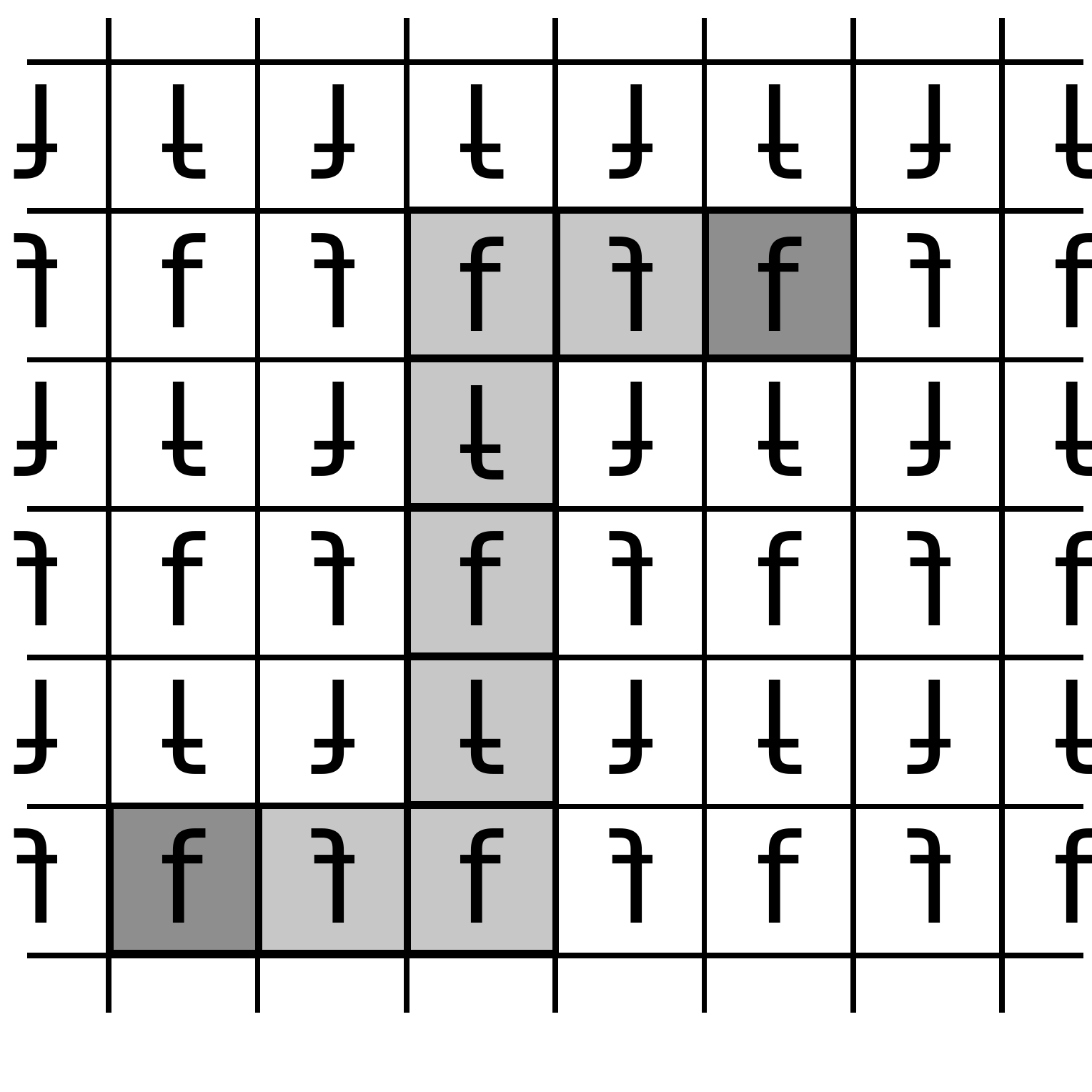}
\centering
\caption{An example of a sequence of reflections sending a function with domain $o$ to a function with domain $x$.
}
\label{Fig:refexample}
\end{figure}
We define the \textit{two-point function},
$$
G_L(x,y) : =  \Big \langle \, F^2_x  \, \, F^2_{y}    \, \, \big ( \prod_{z \in \T_L \setminus \{ x,y\} }F^1_z \big )   \Big \rangle,
$$
omitting the dependence on the functions $F_o^j$ in the notation. For spin system examples we would usually take $F^1_o$ to be the spin at $o$ and $F^2_o=1$, meaning that $G_L(x,y)$ is a spin-spin correlation. We say that the two-point function  is \textit{torus symmetric} if, 
for any $A,B\subset \T_L$ and $z\in \T_L$
\begin{equation}\label{eq:refinvariance}
\big\langle \prod_{x\in A}F^1_x \prod_{x\in B} F^2_x\big\rangle=\big\langle \prod_{x\in A+z}F^1_x \prod_{x\in B+z} F^2_x\big\rangle,
\end{equation}
where the sum is with respect to the torus metric.
As a consequence, for any $x, y, z \in \T_L$, 
\begin{equation}\label{eq:refinvariance2}
G_L(x,y) =  G_L(x + z, y + z), \quad \quad G_L(o,x) =  G_L(-x, o).
\end{equation}
Our first theorem states several site-monotinicity properties for the two-point function. 
\begin{thm}\label{thm:monotonicity}
Consider the torus $\T_L=\Z^d/L\Z^d$ for $d\geq 2$ and $L\in2\N$.
Take $i\in\{1,\dots,d\}$. Suppose that $\langle\cdot\rangle$ is reflection positive for reflections through edges and that the two-point function is torus symmetric. For any $z=(z_1,\dots,z_d)$,
\begin{align}
 G_L(o,z)   &\leq   G_L(o,  z_i \boldsymbol{e}_i)   & \mbox{ if $z_i$ odd} 
 \label{eq:oddinequality}
\\
G_L(o,z)  & \leq \frac{1}{2} \Big ( G_L \big (o, \boldsymbol{e}_i (z_i - 1) \Big ) 
+  G_L \big (o, \boldsymbol{e}_i (z_i + 1) \big )  \Big ) & \mbox{ if $z_i$ even}  \label{eq:eveninequality}
\end{align}
Further, for $y\in\T_L$ such that $y\cdot\be_i=0$ (possibly $y=o$) the function
\begin{equation}\label{eq:oddmonotonicity}
 G_L \big (o,  y + n \boldsymbol{e}_i \big )  +  G_L \big (o,  
n \boldsymbol{e}_i \big )
\end{equation}
is a non-increasing function of $n\in(0,L/2)\cap 2\N+1$.
If, in addition, $\langle\cdot\rangle$ is reflection positive for reflections through vertices then $(\ref{eq:oddinequality})$ also holds for $z_i$ even and (\ref{eq:oddmonotonicity}) holds for any $n\in(0,L/2]$.
\end{thm}

Our next theorem is a consequence of Theorem \ref{thm:monotonicity} and consists of the following statements. 
 Suppose that the two-point function is uniformly bounded from above by a constant $M$,
(i) Whenever the Ces\`aro sum of the two-point function is uniformly positive, the two-point function is\textit{ point-wise} uniformly positive on cartesian axes. (ii) - (iii) If the uniformly positive lower bound to the Ces\`aro sum is close enough to $M$, then the two-point function is point-wise uniformly positive not only on the cartesian axes, but also at any  site in a box centred at the origin whose side length is of order $O(L)$.

\begin{thm}\label{thm:positivity}
Consider the torus $\T_L=\T^d/L\Z^d$ for $d\geq 2$ and $L\in2\N$.
Take $i\in\{1,\dots,d\}$. Suppose that $\langle\cdot\rangle$ is reflection positive for reflections through edges and that the two-point function is torus symmetric. Moreover, suppose that for some $C_1>0$ we have
\begin{equation}\label{eq:FSS}
\liminf_{\substack{L\to\infty\\ L\text{ even}}}\frac{1}{|\T_L|}\sum_{x\in\T_L}
G_L(o,x) \, \geq  \, C_1>0,
\end{equation}
and that  for  some $M\in(0,\infty)$ we have that,
\begin{equation}\label{eq:corrrequirement}
\forall L \in 2 \mathbb{N} \quad \forall x, y \in \T_L \quad G_L(x,y) \leq M.
\end{equation}
Then, the following properties hold,
\begin{enumerate}[(i)]
\item For any $\varphi \in (0, \frac{C_1}{2})$ there exists  $\varepsilon > 0$  such that for any  integer $n \in (- \varepsilon \, L, \varepsilon L )$ and  any $i\in\{1,\dots,d\}$, 
$$
G_L(o, \boldsymbol{e}_i n ) \geq  \varphi. $$
\item For $\varepsilon\in(0,\tfrac12)$ and $L \in 2 \mathbb{N}$ sufficiently large, for any $x\in\T_L$ such that $|x\cdot\be_i|\in(0,\varepsilon L)\cap (2\N+1)$ for every $i\in\{1,\dots,d\}$,
$$
G_L(o,x) \geq M-\big(\tfrac14-\tfrac12\varepsilon\big)^{-d}(M-C_1).
$$
\item If $\langle\cdot\rangle$ is also reflection positive for reflections through vertices then 
 for any $\varepsilon\in(0,\tfrac12)$ and $L \in 2 \mathbb{N}$ sufficiently large, 
for all $x\in\T_L$ such that $|x\cdot\be_i| \in (0, \varepsilon L)$ for every $i\in\{1,\dots,d\}$,
$$
G_L(o,x)\geq M-\big(\tfrac12-\varepsilon\big)^{-d}(M-C_1).
$$
\end{enumerate}
\end{thm}
\begin{rem} \label{remark}
\begin{enumerate}[(i)]
\item  For many statistical mechanics models one has that there exists some positive $c > 0$ such that, if $x$ and $y$ are nearest neighbours, then $G_L(o,x) \geq  G_L(o,y) \, \, c$. When such a property is fulfilled, 
the properties of point-wise positivity of the two-point function stated in (i) and (ii) can be extended to vertices which are not necessarily odd.
\item If we do not care about the size of the box around $o$ where we can show that two-point functions are uniformly bounded then we can simple look at the limit $\varepsilon\to 0$. In this case the bound in (ii) becomes $M-4^d(M-C_1)$ and the bound in (iii) becomes $M-2^d(M-C_1)$.
\end{enumerate}
\end{rem}

\section{Applications}\label{sec:examples}

\subsection{Quantum Heisenberg model}\label{sec:quantumspin}

For $S\in \tfrac12\N$ we define $\Sigma_x=\C^{2S+1}$ and $\otimes$ to be the tensor product, hence
$
\Scal=\otimes_{x\in\T_L}\C^{2S+1}$.
 Let $S^1,S^2,S^3$ denote the spin-$S$ operators on $\C^{2S+1}$. They are hermitian matrices defined by \begin{equation}
[S^1,S^2]=iS^3,\qquad [S^2,S^3]=iS^1,\qquad [S^3,S^2]=iS^2,
\end{equation}
\begin{equation}
(S^1)^2+(S^2)^2+(S^3)^2=S(S+1)\1,
\end{equation}
where $\1$ is the identity matrix. Each spin matrix has spectrum $\{-S,-S+1,\dots,S\}$. We denote by $S^i_x=S^i\otimes \1_{\T_L\setminus\{x\}}$ the operator on $\Scal$ that acts as $S^i$ on $\Sigma_x$ and as $\1$ on each $\Sigma_y$, $y\neq x$. For $u\in [-1,1]$ consider the hamiltonian
\begin{equation}
H_u=-2\sum_{\{x,y\}\in\Ecal_L}(S^1_xS^1_y+uS^2_xS^2_y+S^3_xS^3_y).
\end{equation}
The case $u=1$ gives the Heisenberg ferromagnet, $u=-1$ is equivalent to the Heisenberg antiferromagnet, and $u=0$ is the quantum XY model.
For $\beta\geq0$ corresponding to the \emph{inverse temperature} our linear operator is given by the usual Gibbs state at inverse temperature $\beta$. More precisely, for operator $A$ on $(\C^{2S+1})^{\T_L}$ the expectation of $A$ in the Gibbs state is
\begin{equation}
\langle A\rangle =\frac{1}{Z_{u}(\beta)}\tr \,A e^{-\beta H_u}, \qquad Z_u(\beta)=\tr \,e^{-\beta H_u}.
\end{equation}
Take
\begin{equation}
F^1_x=\1_x\quad \text{ and } \quad F^2_x=S^3_x.
\end{equation}
For $u\leq 0$ we have reflection positivity for reflections through edges \cite{F-S-S, K-L-S2,U}. The following theorem is a direct consequence of Theorem \ref{thm:monotonicity}.
\begin{thm}\label{cor:quantummonotonicity}
Let $\beta\geq 0$, $L\in 2\N$, $S\in\tfrac12\N$, $d\geq 2$ and $u\leq 0$. For any $z\in\N\setminus\{0\}$, 
\begin{equation}
\langle S^3_oS^3_z\rangle \leq \begin{cases}\langle S^3_oS^3_{(z\cdot\be_i)\be_i}\rangle & \text{if }z\cdot\be_i\in 2\N+1, \\
\tfrac12 \left( \langle S^3_oS^3_{(z\cdot\be_i+1)\be_i}\rangle + \langle S^3_oS^3_{(z\cdot\be_i-1)\be_i}\rangle \right) & \text{if }z\cdot\be_i\in 2\N\setminus \{o\}. \end{cases}
\end{equation}
Further for $y\in\T_L$ such that $y\cdot\be_i=0$ (for example $y=o$) the function 
\begin{equation}
\langle S^3_oS^3_{y+n\be_i}\rangle + \langle S^3_oS^3_{n\be_i}\rangle,
\end{equation}
is a non-increasing function of $n$ for odd $n\in (0,L/2)$.
\end{thm}
We now turn our attention to the consequence of Theorem \ref{thm:positivity}. 
It is known from the famous result of Dyson, Lieb and Simon \cite{D-L-S} and various extensions of this result \cite{K-L-S,K-L-S2,U} that for $d\geq 3$ and $S\in\tfrac12 \N$ there are constants $c_1,c_2>0$ such that for $L\in2\N$ sufficiently large
\begin{equation}\label{eq:uniformpositivityquantum}
\frac{1}{|\T_L|}\sum_{x\in\T_L}\langle S^3_oS^3_x\rangle\geq c_1-\frac{c_2}{\beta}.
\end{equation}
Our next theorem extends such a result by showing that the two-point function is \textit{point-wise} uniformly positive on vertices whose coordinates are all odd.

\newpage 

\begin{thm}\label{prop:thm2.8} 
Suppose that $d\geq 3$ and $u\leq 0$.
\begin{enumerate}[(i)]
\item For any $\varphi \in (0, \frac{c_1}{2})$ there exists $\beta$  large enough and $\epsilon > 0$ such that,  for any $L \in  
2 \mathbb{N}$, any odd integer $n\in(-\varepsilon L,\varepsilon L)$ and any $i\in\{1,\dots,d\}$,
\begin{equation}
\langle S_o^3S_{n\be_i}^3\rangle \geq \varphi.
\end{equation}
\item
There exists an explicit $Q(d, u) \in (0 , \infty)$ such that if $S > Q(d, u)$
and $\beta$ is large enough, then there exists $\varphi, \varepsilon > 0 $ such that, for any $L \in 2 \mathbb{N}$ and $y \in \T_L$ such that $\|y \|_{\infty} \leq \varepsilon L$ and, for each $i\in\{1,\dots,d\}$, $y\cdot\be_i\in2\N+1$, 
\begin{equation}\label{eq:quantuniformbound}
\langle S_o^3S_{y}^3\rangle \geq \varphi. 
\end{equation}
\end{enumerate}
\end{thm}
In particular, $Q(0, 3)$ can be taken equal to $8$ and $Q(-1, 3)$ can be taken equal to 
$11$. If we could find a constant $c>0$ as in Remark \ref{remark} (i) then we could extend \eqref{eq:quantuniformbound} to all vertices $y$ such that $\|y\|_{\infty}\leq\varepsilon L$.
\begin{proof}
The first claim follows from  (\ref{eq:uniformpositivityquantum}),
and from an immediate application of the claim  (i) in Theorem \ref{thm:positivity}. We now prove the claim (ii).  We start from (\ref{eq:uniformpositivityquantum}), we have $M=S(S+1)/3$. From \cite{U} obtain an explicit expression for $c_1$, 
\begin{equation}
c_1=\frac{S(S+1)}{3}-\frac{1}{\sqrt{2}}\frac{1}{|\T_L|}\sum_{k\in \T_L^*\setminus\{o\}}\sqrt{\frac{\varepsilon_u(k)}{\varepsilon(k)}}
\end{equation}
where $\T_L^*$ is the Fourier dual lattice, $\varepsilon(k)=2\sum_{i=1}^d(1-\cos(k_i))$ and $\varepsilon_u(k)=\sum_{i=1}^d\big[(1-u\cos(k_i))\langle S^1_oS^1_{e_i}\rangle + (u-\cos(k_i))\langle S^2_oS^2_{e_i}\rangle\big]$. Now it is easy to check that $\varepsilon_u(k)\leq \tfrac{S(S+1)}{6}(1-u)\varepsilon(k+\boldsymbol{\pi})$, which gives
\begin{equation}\label{eq:quantumc1}
c_1\geq \frac{S(S+1)}{3}-\frac{\sqrt{1-u}}{2}\sqrt{\frac{S(S+1)}{3}}J_{d,L}
\end{equation}
where
\begin{equation}\label{eq:J}
J_{d,L}=\frac{1}{|\T_L|}\sum_{k\in \T_L^*\setminus\{o\}}\sqrt{\frac{\varepsilon(k+\boldsymbol{\pi})}{\varepsilon(k)}}
\end{equation}
satisfies $\lim_{d\to\infty}\lim_{L\to\infty}J_{d,L}=1$. Further $\lim_{L\to\infty}J_{d,L}$ is a decreasing function of $d$ and $\lim_{L\to\infty}J_{3,L}=1.15672\cdots$.
Using these bounds, the inequality (ii) of Theorem \ref{thm:positivity} shows that there is some $\varphi>0$ such that for any $x\in\T_L$ with $|x\cdot\be_i|\in(0,\varepsilon L)\cap 2\N+1$ for every $i\in\{1,\dots,d\}$ we have $\langle S_o^3S_x^3\rangle \geq \varphi$ once $\beta$ is sufficiently large if
\begin{equation}
S^2+S-\tfrac34 (1-u)(J_{d,L})^2\big(\tfrac14-\tfrac12\varepsilon\big)^{-2d}>0,
\end{equation}
which is fulfilled for any large enough $S$. This completes the proof. 
 \end{proof}

\subsection{The Random Path Model}\label{sec:rpmpath}
The Random Path Model (RPM) was introduced in \cite{L-T}.
It can be viewed as a random loop model with an arbitrary number of 
coloured loops and walks, with loops and walks possibly sharing the same edge and, at every vertex, a pairing function which pairs pairs of links touching that vertex or leaving them unpaired.
It was shown in \cite{L-T} that, for different choices of the parameters of the RPM, we can obtain many interesting models such as the loop $O(N)$ model, the spin $O(N)$ model, the dimer and double-dimer model and random lattice permutations. 
Here we introduce the  RPM in a  more general setting than in \cite{L-T}. Such a generalisation consists of allowing pairings 
of links with different colours and allows us to derive site monotonicity properties for a more general class of two-point functions, for example, for the probability that a loop connects two distinct vertices of the torus.


Let $\mathcal{E}_L$ be the set of edges connecting nearest neighbour vertices of the torus.
Let $m=(m_e)_{e\in\Ecal_L}\in\N^{\Ecal_L}$ be an assignment of a number of \emph{links} on each edge of $\Ecal_L$ and, for $N\in N_{>0}$, let $c(m)\in\bigtimes_{e\in\Ecal_L}\big(\{1,\dots,N\}^{m_e}\big)$ be a function, which we call a \emph{colouring}, that for each $e\in\Ecal_L$ assigns the $m_e$ links on $e$ with a colour in $\{1,\dots,N\}$. Lastly we define $\pi(m,c(m))=(\pi_x(m,c(m)))_{x\in\T_L}$  consisting of a collection of partitions of links. $\pi_x(m,c(m))$ is a partition of the links incident to $x$ into sets with at most two links each. If, for some $x\in\T_L$, two links are in the same element of the partition at $x$ then we say the links are \emph{paired at $x$} and call this element a \emph{pairing}. If a link is not paired to any other link at $x$ then we say $x$ is \emph{unpaired at $x$}. Links can be paired or unpaired at both end points of their corresponding edge. We denote by $\Wcal_L$ the set of all such triples $(m,c(m),\pi(m,c(m))$ and refer to elements $w=(m(w),c(w),\pi(w))\in\Wcal_L$  as \emph{configurations}. Configurations can be interpreted as a collection of multicoloured loops and walks on $(\T_L,\Ecal_L)$. 

Now for $x\in\T_L$ and $i\in\{1,\dots,N\}$ let $u^i_x$ be the number of unpaired links of colour $i$ at $x$, let $K_x$ be the number of pairings at $x$ between two differently coloured links, and let $n_x$ be the number of elements of $\pi_x$. If $K_x=0$ we define $v^i_x$ to be the number of pairings at $x$ between links with colour $i$, otherwise we define $v^i_x=0$. Finally let $t_x$ be the number of pairings at $x$ between links on the same edge (this is required to recover, for example, the spin $O(N)$ model from the RPM).

Let $U:\N^{2N+3}\to \R$ and $\beta\geq 0$. We define our measure $\mu_{L,N,\beta,U}$ on $\Wcal_L$ as
\begin{equation}
\mu_{L,N,\beta,U}(w)=\prod_{e\in\Ecal_L}\frac{\beta^{m_e(w)}}{m_e(w)!}\prod_{x\in\T_L}U_x(w)\qquad \forall w\in\Wcal_L
\end{equation}
where $U_x(w)=U(u^1_x,\dots,u^N_x,v^1_x,\dots,v^N_x,K_x,n_x,t_x)$. We refer to $U$ as a vertex \emph{weight function}. For $f:\Wcal_L\to\R$ we use the same notation for the expectation of $f$, $\mu_{L,N,\beta,U}(f):=\sum_{w\in\Wcal_L}f(w)\mu_{L,N,\beta,U}(w)$.

The measure $\mu_{L, N, 
\beta, U  }$ was proven to be reflection positive for reflections through edges in \cite[Proposition 3.2]{L-T}. The same result holds for the more general 
random path model defined in this note, since allowing pairing of links with different colour does not modify the proof.

It can be shown that the random path model fits the general framework introduced in the present note, by considering local state spaces for $x\in\T_L$ that consist of a specification of the number of coloured links on each edge incident to $x$ (an element of $\N^{2dN}$) together with a function that maps $\N^{2dN}$ to partitions of $\sqcup_{m\geq0}\{1,\dots,m\}$. The measure is then supported on configurations whose functions partition the correct value of $m$ (the value corresponding to the total number of incident links) at each $x\in\T_L$ and which, for each $e\in\Ecal_L$ specify the same link numbers on $e$ for both end points of $e$.

Suppose that $U_x(w)=0$ whenever $K_x\neq 0$, then $\mu_{L,N,\beta,U}$ is supported on configurations of monochromatic loops and walks. From this we can recover the RPM introduced in \cite{L-T} which reduces to the specific examples mentioned above if we further specify $U$ in an appropriate way. In this case we could take
\begin{equation}
\langle\cdot\rangle=\frac{1}{Z^{loop}_{L,N,\beta,U}}\mu_{L,N,\beta,U}(\cdot)
\end{equation}
where $Z^{loop}_{L,N,\beta,U}$ is the total measure under $\mu_{L,N,\beta,U}$ of configurations with only loops. We then take
\begin{equation}
F^1_x=\1_{u^1_x=0}\qquad \text{ and } F^2_x=\1_{u^1_x=1}
\end{equation}
and find that $G_L(x,y)$ corresponds to the two-point function introduced in \cite{L-T}, when $U$ is chosen appropriately this is equal to the spin-spin correlation of the spin $O(N)$ model. From this we can recover Theorems 2.4, 2.6 and 2.8 in \cite{L-T} .

Now suppose that $N > 1$, that $U_x$ allows links of different colours to be paired, and that it is $0$ if $\sum_i u^i_x\neq 0$ (meaning the model only has loops and no walks). Our linear functional $\langle\cdot\rangle$ could then be given by 
\begin{equation}
\langle\cdot\rangle=\frac{1}{Z^{mono}_{L,N,\beta,U}}\mu_{L,N,\beta,U}(\cdot)
\end{equation}
where $Z^{mono}_{L,N,\beta,U}$  is the total measure under $\mu_{L,N,\beta,U}$ of configurations with $\sum_xK_x=0$ and only loops.
Now we take 
\begin{equation}
F^1_x=\1_{K_x=0}\quad \text{ and }\quad F^2_x=\1_{K_x=1}.
\end{equation}
We have that $G_L(x,y)=2\binom{N}{2}\P(x\leftrightarrow y)$ where the probability is in the system with only monochromatic loops with colours in $\{1,\dots,N\}$ and there are no walks. 
The event $x\leftrightarrow y$ is the event that there is a loop that passes through $x$ and $y$.

  Theorem \ref{thm:monotonicity} leads then to the following theorem. 
\begin{thm}
Let  $\mathbb{P}(  x \leftrightarrow y  )$ be the probability that 
a loop passes through $x$ and $y$ in the random path model with only monochromatic loops and no open path.
 For any $z=(z_1,\dots,z_d)$,
\begin{align}
\P(o\leftrightarrow z)&\leq \P(o\leftrightarrow z_i\be_i) \qquad \qquad\qquad\qquad\qquad\quad \,\,& \text{ if }  z_i\in2\Z+1,
\\
\P(o\leftrightarrow z)&\leq \tfrac12\P(o\leftrightarrow (z_i-1)\be_i)+\tfrac12 \P(o\leftrightarrow (z_i+1)\be_i) \, & \text{ if } z_i\in 2\Z\setminus \{0\},
\end{align}
and that for $y\in\T_L$ such that $y\cdot\be_i=0$ 
\begin{equation}
\P(o\leftrightarrow y+n\be_i)+\P(o\leftrightarrow n\be_i)
\end{equation}
is a non-increasing function of $n$ for all odd $n\in(0,L/2)$.
\end{thm} 
 Note that $\mathbb{P}(  x \leftrightarrow y  )$ equals the probability that a  loop connects $x$ and $y$ in the loop O(N) model, in the double dimer model, in lattice permutations or in the loop representation of the spin O(N) model under an appropriate choice of $U$ \cite{L-T}.
Further, it has been proven \cite{B-U} that, when $U$ is chosen appropriately, such a probability equals the following correlation, $\mathbb{P}(  x \leftrightarrow y ) = \langle S_x^1 S_x^2 S_y^1 S_y^2   \rangle$, in the spin O(N) model with $N> 1$, hence our theorem provides monotonicity properties for such a four-spin correlation function.

\section{Proof of Theorem \ref{thm:monotonicity}}\label{sec:proofmonotonicity}

Suppose that $\langle\cdot\rangle$ is reflection positive with respect to the reflection $\vartheta$. Let $Q\subset \T_L$ and define $Q^{\pm}:=(Q\cap \T_L^{\pm})\cup\vartheta(Q\cap\T_L^{\pm})$. The key to the proof is the following lemma.

\begin{lem} \label{lem:keylem}
For $Q\subset\T_L$
\begin{equation}
\sum_{\substack{x,y\in Q\\ x\neq y}}G_L(x,y)\leq \frac12 \sum_{\substack{x,y\in Q^+\\ x\neq y}}G_L(x,y)+\frac12\sum_{\substack{x,y\in Q^-\\ x\neq y}}G_L(x,y).
\end{equation}
\end{lem}
\begin{proof}
For $0<\eta \ll 1$ we consider the following functions
\begin{equation}
A=\prod_{x\in Q\cap\T_L^+}(1+\eta F^2_x\prod_{z\in\T^+_L\setminus\{x\}}F^1_z),\qquad B=\prod_{x\in Q\cap\T_L^-}(1+\eta F^2_{\vartheta x}\prod_{z\in\T^-_L\setminus\{ x\}}F^1_{\vartheta z}).
\end{equation}
Now for simplicity of notation we write $\T_L(x)$ for $\T_L^+\setminus\{x\}$ if $x\in \T_L^+$ and $\T_L^-\setminus\{x\}$ if $x\in \T_L^-$.
A simple calculation gives
\begin{equation}
\begin{aligned}
\langle A\vartheta B\rangle&=\big\langle\prod_{x\in Q}\big(1+\eta F^2_x\prod_{z\in\T_L(x)}F^1_z\big)\big\rangle
\\
&=1+\eta\sum_{x\in Q}\big\langle F^2_x\prod_{z\in\T_L(x)}F^1_z\big\rangle+\eta^2\sum_{\substack{x,y\in Q\\ x\neq y}}\big\langle F^2_xF^2_y\prod_{z\in\T_L(x)}F^1_z\prod_{z\in\T_L(y)}F^1_z\big\rangle +O(\eta^3),
\end{aligned}
\end{equation}
and analogously
\begin{align}
\langle A\vartheta A\rangle&=1+\eta\sum_{x\in Q^+}\big\langle F^2_x\prod_{z\in\T_L(x)}F^1_z\big\rangle+\eta^2\sum_{\substack{x,y\in Q^+\\ x\neq y}}\big\langle F^2_xF^2_y\prod_{z\in\T_L(x)}F^1_z\prod_{z\in\T_L(y)}F^1_z\big\rangle +O(\eta^3),
\\
\langle B\vartheta B\rangle&=1+\eta\sum_{x\in Q^-}\big\langle F^2_x\prod_{z\in\T_L(x)}F^1_z\big\rangle+\eta^2\sum_{\substack{x,y\in Q^-\\ x\neq y}}\big\langle F^2_xF^2_y\prod_{z\in\T_L(x)}F^1_z\prod_{z\in\T_L(y)}F^1_z\big\rangle +O(\eta^3).
\end{align}
Now suppose that $x,y\in Q\cap\T_L^+$, then $x,y,\vartheta x,\vartheta y\in Q^+$ and we further note that 
\begin{equation}\label{eq:mixedterms}
\big\langle F^2_xF^2_y\prod_{z\in\T_L(x)}F^1_z\prod_{z\in\T_L(y)}F^1_z\big\rangle=\big\langle F^2_{\vartheta x}F^2_{\vartheta y}  \prod_{z\in\T_L(\vartheta x)}F^1_z\prod_{z\in\T_L(\vartheta y)}F^1_z\big\rangle.
\end{equation}
 An analogous identity holds for $x,y\in Q\cap \T_L^-$.
Now we use \eqref{eq:refpos}. Note that the $\eta$ terms will cancel by \eqref{eq:refinvariance}. 
Now we compare the $\eta^2$ terms. The terms $\big\langle F^2_xF^2_y\prod_{z\in\T_L(x)}F^1_z\prod_{z\in\T_L(y)}F^1_z\big\rangle$ when $x,y\in Q\cap\T_L^{\pm}$ will cancel due to \eqref{eq:mixedterms}. By using \eqref{eq:refinvariance} repeatedly on the remaining terms to group those terms that are equal gives the result.
\end{proof}

We take $Q=\{o,z\}$ and $\vartheta$ the reflection in the plane bisecting $\{p\be_i,(p+1)\be_i\}$ for $p:=\tfrac12 (z\cdot \be_i-1+q\}$, this requires $z\cdot\be_i+q\in2\N+1$ and $z\cdot\be_i\pm q\in(0,L)$. If we take $q=0$ when $z_i\in 2\N+1$ and $q=1$ when $z_i\in2\N\setminus\{0\}$ then Lemma \ref{lem:keylem} gives us \eqref{eq:oddinequality} and \eqref{eq:eveninequality}.
If we also have reflection positivity for reflections through sites then we can reflect in the plane $R=\{x\in\R\,:\, x\cdot\be_i=\tfrac12(z\cdot\be_i+q)\}$, requiring that $z\cdot\be_i+q$ is even. If we apply Lemma \ref{lem:keylem} with $q=0$ we find that for $z\cdot\be_i\in2\N\setminus\{0\}$ we also have \eqref{eq:oddinequality}.

For the monotonicity result \eqref{eq:oddmonotonicity} we take $Q=\{o,z,z_i\be_i,z-z_i\be_i\}$ with the same reflection as above. We define the function
\begin{equation}
G^{\be_i}_L(x):=\tfrac12\big(G_L(o,x)+G_L(o,(x\cdot\be_i)\be_i)\big),
\end{equation}
and find, using Lemma \ref{lem:keylem}, after rearranging and  \eqref{eq:refinvariance} that for $z_i+q$ odd
\begin{equation}\label{eq:iterativeinequality}
G^{\be_i}_L(z+q\be_i)-G^{\be_i}_L(z)\geq G^{\be_i}_L(z)+G^{\be_i}_L(z-q\be_i).
\end{equation}
The proof follows the proof of \cite[Proposition 4.2]{L-T}.
We can now prove \eqref{eq:oddmonotonicity} by contradiction. Suppose that $y\in\T_L$ such that $y\cdot\be_i=0$ and odd $n\in(0,L/2)$ satisfy $G^{\be_i}_L(y+n\be_i)>G^{\be_i}_L(y+(n-2)\be_i)$. Now by repeatedly using \eqref{eq:iterativeinequality} with $q=2$ we find 
\begin{equation}
G^{\be_i}_L(y+n\be_i)>G^{\be_i}_L(y+(n-2)\be_i)>G^{\be_i}_L(y+(n-4)\be_i)>G^{\be_i}_L(y+(n-6)\be_i)\dots
\end{equation}
Once we have used this inequality $n$ times we find $G^{\be_i}_L(y+n\be_i)>G^{\be_i}_L(y+n\be_i-2n\be_i)=G^{\be_i}_L(y-n\be_i)$, but by reflection positivity we must have  $G^{\be_i}_L(y-n\be_i)=G^{\be_i}_L(y+n\be_i)$. This contradiction completes the proof of \eqref{eq:oddmonotonicity}. If, in addition, we have reflection positivity for reflections through sites we can use the reflection in $R=\{x\in\R\,:\, x\cdot\be_i=\tfrac12(z\cdot\be_i+q)\}$. We then obtain the inequality \eqref{eq:iterativeinequality} for $z_i+q$ even. Using this we can obtain a contradiction as before by alternating between the odd and even version of \eqref{eq:iterativeinequality} with $q=1$ to find that for any $y\in\T_L$ such that $y\cdot\be_i\pm 1\in(0,L)$
\begin{equation}
G^{\be_i}_L(y+\be_i)-G^{\be_i}_L(y)\geq G^{\be_i}_L(y)-G^{\be_i}_L(y-\be_i).
\end{equation}
The full monotonicity result then follows similarly to \eqref{eq:oddmonotonicity}.

\section{Proof of Theorem \ref{thm:positivity}}\label{sec:proofpositivity}
We start with the proof of (i) and we present the proof of (ii) and (iii) afterwards.
To begin, fix an arbitrary $\varphi \in (0, C_1)$. We claim that there must exist an $\epsilon > 0$ small enough such that for any $L \in 2 \mathbb{N}$ there exists $z_L \in  \T_L \setminus [0, \epsilon L]^d $ such that  $G_L(o,x) \geq \varphi$. 
The proof of this claim is by contradiction. Suppose that this was not the case, then, under the assumptions of the theorem, we would have that 
$$
\sum_{x \in \T_L} G_L(o,x) \leq \, \,  \varphi \, \, \lceil  (\, 1 \, - \, \epsilon  \, ) \, L \rceil ^d \, + \, M  \lceil \epsilon L \rceil^d,
$$
which would be in contradiction with (\ref{eq:FSS}) for small enough $\epsilon$, since we assumed that $\varphi< C_1$.
Now define  $y_L : = z_L \cdot \boldsymbol{e}_1$ and, if it is  odd,  we use the first claim in Theorem \ref{thm:monotonicity} and  deduce that, 
$
G_L \big (o, y_L \boldsymbol{e}_1 \big ) \, \, \geq \, \, \varphi,
$
otherwise we use the second claim in Theorem \ref{thm:monotonicity} and deduce that, 
$
\max \big \{G_L \big (o, (y_L + 1) \boldsymbol{e}_1  \big ),   G_L \big (o, (y_L - 1) \boldsymbol{e}_1  \big )  \big \} \, \, \geq \, \, \frac{\varphi}{2}.
$
Using the fact that $y_L+ 1 \geq \epsilon L$ and the last claim in Theorem \ref{thm:monotonicity}, we deduce that, for any odd integer in the interval $n \in (o, \epsilon L)$,
$
G_L \big (o, n  \boldsymbol{e}_1  \big ) \geq \frac{\varphi}{2}.
$
This concludes the proof of (i). 
We now proceed with the proof of (ii) and (iii).
To begin, for $z\in\T_L$ we define
\begin{equation}
\Q_z:=\{(x_1,\dots,x_d)\in\Z^d\,:\, \forall i\in\{1,\dots,d\},\, x_i\leq |z\cdot\be_i| \text{ or } x_i>L-|z\cdot\be_i]\}.
\end{equation}
The proof relies on the following lemmas.

\begin{lem}\label{lem:keylem2}
Let $z\in\T_L$ and $y\in\mathbb{Q}_z$ be such that $z_i$ and $y_i$ are odd for every $i\in\{1,\dots,d\}$ then under the same assumptions as Theorem \ref{thm:positivity}
\begin{equation}
G_L(o,y)\geq 2^dG_L(o,z) -(2^d-1)M.
\end{equation}
If, in addition, $\langle\cdot\rangle$ is reflection positive for reflections through vertices then the inequality holds for any $z\in\T_L$ and $y\in\Q_z$.
\end{lem}
\begin{proof}
The proof is as in the proof of \cite[Proposition 4.7]{L-T} with minor changes as we only have the monotonicity result \eqref{eq:oddmonotonicity} for odd $n$. For convenience we assume that $z_i,y_i>0$ for every $i\in\{1,\dots,d\}$, other cases follow by symmetry. For $i\in\{1,\dots, d\}$ define
\begin{equation}
D_i:=(z-y)\cdot \be_i,
\end{equation}
then $D_i\in 2\N$. There is a ``path"
\begin{equation}
(z^1_0,z^1_1,\dots,z^1_{D_1/2},z^2_0,z^2_1,\dots,z^2_{D_2/2},\dots,z^d_0,z^d_1,\dots,z^d_{D_d/2})
\end{equation}
with the properties that $z^1_0=z$, $z^d_{D_d/2}=y$, and, for every  $i\in\{1,\dots,d-1\}$, $z^i_{D_i/2}=z^{i+1}_1$. Further, for each $i\in\{1,\dots,d\}$ and $j\in [1,D_i/2]$
\begin{equation}
z^i_{j-1}-z^i_j=2\be_i.
\end{equation}
Now we use both \eqref{eq:oddinequality} and \eqref{eq:oddmonotonicity},
\begin{equation}
\begin{aligned}
2G_L(o,z^i_0) &\leq G_L(o,z^i_0) +G_L(o,(z^i_0\cdot\be_i)\be_i)
\\
 &\leq G_L(o,z^i_{D_i/2}) + G_L(o,(z^i_{D_i/2}\cdot\be_i)\be_i),
\end{aligned}
\end{equation}
and hence using that $G_L(o,x)\leq M$ for any $x\in\T_L$  we have that
\begin{equation}
G_L(o,z^i_{D_i/2}) \geq 2G_L(o,z^i_0)-M.
\end{equation}
Iterating this for $i=1,\dots, d$ gives
\begin{equation}
\begin{aligned}
G_L(o,y)=G_L(o,z^d_{D_d/2})&\geq 2G_L(o,z^d_0)-M\geq \dots
\\
& \geq  2^dG_L(o,z)-(2^d-1)M,
\end{aligned}
\end{equation}
this completes the proof. If $\langle\cdot\rangle$ is also reflection positive for reflections through vertices the proof is exactly as in \cite[Proposition 4.7]{L-T}. We define $D_i$'s and the path $(z_0^1,\dots z^d_{D_d/2}$ as before except that we can take $z^i_{j-1}-z^i_j=\be_i$, the rest of the proof then proceeds as before.
\end{proof}
Now, for $r\in \N$ let
\begin{equation}
\mathbb{S}_{r,L}:=\{z\in\T_L\, :\, \exists i\in\{1,\dots,d\}\text{ such that } z\cdot \be_i<r\text{ or }L-z\cdot\be_i\leq r\}.
\end{equation}

\begin{lem}\label{lem:keylem3}
Under the same assumptions as \ref{thm:positivity} there are $x_L\in \T_L\setminus\mathbb{S}_{\varepsilon L,L}$ and  $z_L\in\T_L\setminus\mathbb{S}_{\varepsilon L,L}$ with $|z_L\cdot\be_i|\in 2\N+1$ for every $i\in\{1,\dots,d\}$ such that
\begin{align}
G_L(o,x_L)&\geq M-(1-2\varepsilon)^{-d}(M-C_1), \label{eq:xL}
\\
G_L(o,z_L)&\geq M-\big(\tfrac12-\varepsilon\big)^{-d}(M-C_1) \label{eq:zL}.
\end{align}
\end{lem}
\begin{proof}
The proof of \eqref{eq:xL} is exactly as in \cite[Lemma 4.9]{L-T}. The proof of \eqref{eq:zL} is a simple adaptation of \cite[Lemma 4.9]{L-T} and we sketch it here. 
Now a simple proof by contradiction shows that there must be a $z_L$ as in the statement of the lemma. Indeed, suppose for every $z_L\in\T_L$ with $|z_L\cdot\be_i|\in [\varepsilon L,L)\cap2\N+1$ for every $i\in\{1,\dots,d\}$ that $G_L(o,z_L)< M-\big(\tfrac12-\varepsilon\big)^{-d}(M-C_1)$. Using this together with the worst-case bound $M$ for every other vertex and the bound $|\T_L\setminus \mathbb{S}_{r,L}|=(L-2r)^d$ gives a contradiction.
\end{proof}

Statement (i) of Theorem \ref{thm:positivity} follows immediately from \eqref{eq:xL} and Theorem \ref{thm:monotonicity}. For statement (ii) of Theorem \ref{thm:positivity} note that if $z_L$ is as in the statement of Lemma \ref{lem:keylem3} then, by Lemma \ref{lem:keylem2}, for any $y\in\mathbb{Q}_{z_L}$ such that $y_i$ is odd for each $i\in\{1,\dots,d\}$ we have (after rearranging)
\begin{equation}
G_L(o,y)\geq 2^dG_L(o,z_L)-(2^d-1)M\geq M-2^d\big(\tfrac12-\varepsilon\big)^{-d}(M-C_1).
\end{equation}
which is equal to the bound in the Theorem.
Finally for statement (iii) of Theorem \ref{thm:positivity} we note that by Lemmas \ref{lem:keylem2} and \ref{lem:keylem3} for any $y\in\Q_{x_L}$ we have (after rearranging)
\begin{equation}
G_L(o,y)\geq 2^dG_L(o,x_L)-(2^d-1)M\geq M-2^d(1-2\varepsilon)^{-d}(M-C_1).
\end{equation}

\end{document}